\documentclass[a4paper,twoside,12pt]{article}
\usepackage{amssymb,amsmath,amsthm,latexsym}
\usepackage{amsfonts}
\usepackage{graphicx,caption,subcaption}
\usepackage[pdftex,bookmarks,colorlinks=false]{hyperref}
\usepackage[hmargin=1.2in,vmargin=1.2in]{geometry}
\usepackage[mathscr]{euscript}
\usepackage{float}
\usepackage[auth-lg,affil-it]{authblk}
\numberwithin{equation}{section}

\allowdisplaybreaks

\def\ni {\noindent}

\newcommand{\N}{\mathbb{N}}
\newcommand{\C}{\mathbb{C}}

\newtheorem{thm}{Theorem}[section]

\newtheorem{lem}[thm]{Lemma}

\newtheorem{cor}[thm]{Corollary}

\title{\textbf{A Note on Chromatic Blending of Colour Clusters}}

\vspace{0.5cm}
\author{Johan Kok}
\affil{\small Tshwane Metropolitan Police Department, \\ City of Tshwane, South Africa\\ {\tt kokkiek2@tshwane.gov.za}}

\author{Sudev Naduvath}
\affil{\small Centre for Studies in Discrete Mathematics\\ Vidya Academy of Science \& Technology\\ Thrissur, India.\\ {\tt sudevnk@gmail.com}}

\author{Muhammad Kamran Jamil}
\affil{\small Department of Mathematics\\ Riphah Institute of Computing and Applied Sciences\\ Riphah International University\\ Lahore, Pakistan.\\ {\tt m.kamran.sms@gmail.com}}

\date{}

\begin{document}
\maketitle

\begin{abstract}
\ni For a colour cluster $\C =(\mathcal{C}_1,\mathcal{C}_2, \mathcal{C}_3,\dots,\mathcal{C}_\ell)$, $\mathcal{C}_i$ is a colour class, and $|\mathcal{C}_i|=r_i \geq 1$, we investigate a simple connected graph structure $G^{\C}$, which represents a graphical embodiment of the colour cluster such that the chromatic number $\chi(G^{\C})= \ell,$ and the number of edges is a maximum, denoted $\varepsilon^+(G^{\C})$. We also extend the study by inducing new colour clusters recursively by blending the colours of all pairs of adjacent vertices. Recursion repeats until a maximal homogeneous blend between all $\ell$ colours is obtained. This is called total chromatic blending. Total chromatic blending models for example, total genetic, chemical, cultural or social orderliness integration.  
\end{abstract}

\ni \textbf{Keywords:} Graphical embodiment, colour cluster, colour classes, edge-maximality, chromatic blending.

\vspace{0.25cm}

\ni \textbf{Mathematics Subject Classification:} 05C15, 05C38, 05C75, 05C85.
 
\section{Introduction}

For general notation and concepts in graphs and digraphs see \cite{BM1,CL1,FH,DBW}. For colouring concepts in graph theory, please see \cite{CZ1,JT1,MK1}. Unless stated otherwise, all graphs mentioned in this paper are simple, connected  and undirected graphs.

It is well known that the chromatic number $\chi(G)\geq 1$ of a graph $G$ is the minimum number of distinct colours that allow a proper colouring of $G$. Such a colouring of $G$ is called a \textit{chromatic colouring} of $G$. In \cite{KSC}, the number of times a colour $c_j$ has been allocated is called the \textit{colour weight} of the colour $c_i$ and is denoted by $\theta(c_i) \geq 1$. Hence, the colours can be partitioned into non-empty colour classes $\mathcal{C}_i$, $1\leq i\leq \chi(G)$ with each colour class, say $\mathcal{C}_j =(\underbrace{c_j,c_j,c_j,\dots,c_j}_{\theta(c_j)})$. 

Analogous to set theory notation, let $\C = \bigcup\limits_{1\leq i \leq \chi(G)}\mathcal{C}_i$, and be called a colour cluster. It is easy to see that for $\chi(G)\geq 2$ a graph $G^{\C}$ with $\max\varepsilon(G^{\C})$ that requires the colour cluster $\C$ to allow exactly a chromatic colouring (minimum proper colouring) is the complete $\chi(G)$-partite graph $K_{\theta(c_i),\dots,_{\forall c_i}}$ or put differently, the complete $\ell$-partite graph with vertex partitioning $V_i(G^{\C}),\ 1\leq i \leq \ell$ such that $|V_i(G^{\C})| = |\mathcal{C}_i|$ and $\bigcup\limits_{1\leq i \leq \ell} V_i(G^{\C}) = V(G^{\C})$. 

\section{Graphical Embodiment of $\C$}

As stated earlier, for any colour cluster $\C = (\mathcal{C}_1,\mathcal{C}_2,\mathcal{C}_3,\dots,\mathcal{C}_\ell)$ with $|\mathcal{C}_i| = r_i >0$, $\ell\geq 2$, $1\leq i \leq \ell$ the connected graphical embodiment with maximum edges that allows $\C$ a chromatic colouring is the complete $\ell$-partite graph $G^{\C}= K_{r_1,r_2,r_3,\dots,r_\ell}$. The maximum number of edges, denoted by $\varepsilon^+(G^{\C})$, can be expressed in at least two useful ways, as mentioned below. 

If $N=\sum\limits_{i=1}^{\ell}r_i$, then $\varepsilon^+(G^{\C})=\binom{N}{2}-\sum\limits_{i=1}^{\ell}\binom{r_i}{2}$. The second useful way to express the maximum number of edges of $G^{\C}$ is by considering the set $A = \{r_1,r_2,r_3,\dots,r_\ell \}$ and the set of subsets $A'=\{\{r_i,r_j\}: i\neq j, \forall~r_i,r_j\in A\}$. It follows that $\varepsilon^+(G^{\C}) = \sum\limits_{}r_i\cdot r_j$, for all $\{r_i,r_j\}\in A'$.

\ni The following lemma, provided in \cite{KSM1}, is a useful result in our present study.

\begin{lem}
For any colour cluster $\C = (\mathcal{C}_1,\mathcal{C}_2,\mathcal{C}_3,\dots,\mathcal{C}_\ell)$ with $|\mathcal{C}_i| = r_i >0$, $\ell\geq 2$, $1\leq i \leq \ell$ the connected graphical embodiment with minimum edges that allows $\C$ a proper colouring has $\sum\limits_{i=1}^{\ell}r_i -1$ edges. 
\end{lem}

Two types of methodologies for constructing graphical embodiments of colour clusters have been introduced in \cite{KSM1}. In the first type construction, in order to obtain a graphical embodiment $G^{\C}_1$, we first consider the induced star subgraph $\langle v_{1,1},v_{2,1},v_{3,1},\dots,v_{\ell,1}\rangle$ of $G_1$ and then add the edges to obtain a complete subgraph, $K_\ell$. Note that $\frac{1}{2}(\ell -1)(\ell -2)$ edges are to be added in this process. This fact has brought an existence theorem for graphical embodiments of colour clusters as given below.

\begin{thm}
{\rm \cite{KSM1}} For any colour cluster $\C$ there exists at least one graphical embodiment both with minimum number of edges which allows $\C$ to be a chromatic colouring and $\chi(G^{\C}) = \ell$.
\end{thm}

For the colour cluster $\C=(\mathcal{C}_1)$ the graph $G^{\C}$ is a null graph (no edges). Henceforth, a colour cluster will have at least two colour classes unless mentioned otherwise. Note that $\varepsilon^-(G^{\C}) = \sum\limits_{i=1}^{\ell}r_i + \frac{\ell}{2}(\ell -3)$. Hence, we have the following theorem.
 
\begin{thm}
For a colour cluster $\C = (\mathcal{C}_i)$, $2\leq i \leq \ell$ we have $\varepsilon^-(G^{\C}) =\varepsilon^+(G^{\C})$ if and only if $|\mathcal{C}_i|=1$, $\forall~ 1\leq i\leq \ell$. 
\end{thm}
\begin{proof}
For a colour cluster $\C = (\mathcal{C}_i)$, $1\leq i \leq \ell$ and $|\mathcal{C}_i|=1$, $\forall~ 1\leq i\leq \ell$, construction of the Type-I graphical embodiment clearly results in a complete graph, $K_\ell$.  Also, for a complete graph $K_\ell,\ \chi(K_\ell)=\ell$ and $\chi(K_\ell-e)=\chi(K_\ell)-1$ for any edge $e$. Therefore, $\varepsilon^-(G^{\C})=\frac{1}{2}\ell(\ell-1)=\varepsilon(K_\ell)= \varepsilon^+(G^{\C})$. Furthermore,  let for exactly one colour class say, $|\mathcal{C}_j| \geq 2$. Clearly, $\varepsilon^-(G^{\C}) = \frac{1}{2}\ell(\ell-1) + |\mathcal{C}_j| -1$. Since, $\varepsilon^+(G^{\C}) = \frac{1}{2}\ell(\ell-1) + (|\mathcal{C}_j| -1)\cdot \sum\limits_{i=1}^{\ell}r_i$, $i\neq j$ it follows that $\varepsilon^-(G^{\C}) <\varepsilon^+(G^{\C})$. By immediate induction, it follows that for any number of colour classes, say $k,\ 1\leq k \leq \ell$, each with cardinality greater or equal to $2$, we have $\varepsilon^-(G^{\C}) <\varepsilon^+(G^{\C})$.

Conversely, consider the complete graph $K_\ell$, $\ell \geq 2$. It is known that $\chi(K_\ell) = \ell$ hence it corresponds to the colour cluster, $\C = (\mathcal{C}_i)$, $2\leq i \leq \ell$ with $|\mathcal{C}_i|=1$, $\forall~ 1\leq i\leq \ell$.  
\end{proof}

\section{Chromatic Blending}

Consider any colour cluster $\C = (\mathcal{C}_i)$, $1\leq i \leq \ell$. When two or more distinct colours $c_i,c_j,c_k,\dots, c_t$ blend (intuitively understood to be mixing of colours to obtain a distinct new colour), the new colour is written as $c_{i,j,k,\dots,t}$. It is referred to as chromatic blending. Clearly, $c_{i,i,i,\dots,i} = c_i$. Also by convention the number of times a particular colour is added does not change the blend. Therefore, it is understood that $c_{i,j,j,j,\dots,j} = c_{i,j}$. For a graph $G$ that allows $\C$ as a proper colouring, we obtain a new colour cluster by 
\begin{enumerate}\itemsep0mm
\item[(i)] Colouring the edge $uv$ with the blend $c_{i,j}$, $c(u)=c_i$, $c(v)=c_j$.
\item[(ii)] Replace each edge with a new vertex and delete all previous vertices. 
\end{enumerate}

Clearly, a partition of new vertices is obtained which corresponds to a new colour cluster say $\C'$. If for a blended colour class $\mathcal{C}_1 = (c_{j,k,\dots,m},c_{j,k,\dots,m},\dots,c_{j,k,\dots,m})$ the colour $c_s$ is added to each blended colour we may write, $\mathcal{C}_{s,\mathcal{C}_1}$.

We note that for a complete bipartite graph $K_{r_1,r_2}$ all edges are coloured $c_{1,2}$. Hence, the $r_1\cdot r_2$ new vertices are all coloured $c_{1,2}$. This implies that $G^{\C'}$ is a null graph. So on the first iteration an edgeless graph is obtained. The maximum number of edges obtained through the iterations, i.e. iterations 0 and 1, is exactly the number of edges of the initial complete bipartite graph $K_{r_1,r_2}$. We say that the null graph corresponds to total chromatic blending. 

\begin{thm}\label{Thm-3.1}
Consider a colour cluster $\C = (\mathcal{C}_i)$, $|\mathcal{C}_i|=r_i \geq 1$, $1\leq i \leq \ell,\ \ell \geq 2$. Then, the total chromatic blending is obtained on the $t_\chi = (\ell-1)$-th chromatic blending iteration.
\end{thm}
\begin{proof}
For a complete bipartite graph $K_{r_1,r_2}$ all edges are coloured $c_{1,2}$. Hence, the $r_1\cdot r_2$ new vertices are all coloured $c_{1,2}$. This implies that $G^{\C'}$ is a null graph. So on the first iteration an edgeless graph is obtained. The maximum number of edges obtained through the iterations, that is iterations $0$ and $1$, is exactly the number of edges of the initial complete bipartite graph $K_{r_1,r_2}$, (iteration $0$). Hence, the result holds for the complete bipartite graph corresponding to any $\C=(\mathcal{C}_1,\mathcal{C}_2)$. 

Assume that the result holds for the complete $\ell$-partite graph corresponding to any colour cluster $\C = (\mathcal{C}_i)$, $1\leq i \leq \ell$, $2\leq \ell \leq k$. Now, consider the complete $(k+1)$-partite graph corresponding to $\C' = (\C,\mathcal{C}_{k+1})= (\mathcal{C}_i, \mathcal{C}_{k+1})$, $1\leq i \leq k$, $r_{k+1}>0$. Clearly, a maximum number of new edges, $\sum\limits_{i=1}^{k}r_{k+1}\cdot r_i$, are added. It is known that at the $(k-1)$-th iteration, the vertex partition will correspond to the colour cluster $(\underbrace{c_{1,2,3,\dots,k},\dots,c_{1,2,3,\dots,k}}_{\varepsilon(K_{r_1,r_2,r_3,\dots,r_k)}~times},\underbrace{\mathcal{C}_{(k+1),\mathcal{C}_1},\dots,\mathcal{C}_{(k+1),\mathcal{C}_t}}_{finite ~t~ entries})$. Therefore, a new complete $(t+1)$-partite graph is obtained. Hence, for $\ell=k+1$, total chromatic blending is obtained on iteration $t_\chi=k$ because all new vertices are coloured, $c_{1,2,3,\dots,k,(k+1)}$. Hence, by induction, the general result follows for all $\ell \in \N,\ \ell \geq 2$.
\end{proof}

Total chromatic blending models for example, total genetic or chemical or cultural integration. A social dynamical application could be the following. If people belonging to the same gene pool and/or cultural system and/or other, occupy different portions of land and they then integrate the different portions into one democratic country, it will most probably take at least three generations to reach a well integrated \textit{rainbow nation}. This stems from the fact that the chromatic number of the map graph $G$ is $\chi(G)\leq 4$. If the differences are sufficiently distinct that the map graph requires a proper colouring greater than $4$, the number of generations can be noticeably more.

\begin{cor} 
Consider a colour cluster $\C = (\mathcal{C}_i)$, $|\mathcal{C}_i|=r_i \geq 1$, $1\leq i \leq \ell$, $\ell \geq 2$. We have that on the $t_\chi = (\ell-1)$-th chromatic blending iteration the number of vertices of the null graph is obtained recursively as follows.

\begin{enumerate}	\itemsep0mm
\item[(i)] Iteration 1: Consider the complete $\ell$-partite graph to obtain the colours blends

\ni $((\underbrace{c_{1,2},\dots,c_{1,2}}_{r_1\cdot r_2~ entries}),(\underbrace{c_{1,3},\dots,c_{1,3}}_{r_1\cdot r_3~ entries}),\dots, (\underbrace{c_{1,\ell},\dots,c_{1,\ell}}_{r_1\cdot r_\ell~ entries}),(\underbrace{c_{2,3},\dots,c_{2,3}}_{r_2\cdot r_3~ entries}),\dots, (\underbrace{c_{2,\ell},\dots,c_{2,\ell}}_{r_2\cdot r_\ell~ entries}),\\
(\underbrace{c_{3,4},\dots,c_{3,4}}_{r_3\cdot r_4~ entries}),\dots, (\underbrace{c_{3,\ell},\dots,c_{3,\ell}}_{r_3\cdot r_\ell~ entries}),\dots, (\underbrace{c_{(\ell-1),\ell},\dots c_{(\ell-1),\ell}}_{r_{\ell-1}\cdot r_\ell~ entries}))$.\\

\item[(ii)] Iteration 2: String the blended colour cluster of iteration (i) as:\\
$(\underbrace{c_{1,2},\dots,c_{1,2}}_{r_1\cdot r_2~ entries},\underbrace{c_{1,3},\dots,c_{1,3}}_{r_1\cdot r_3~ entries},\dots, \underbrace{c_{1,\ell},\dots,c_{1,\ell}}_{r_1\cdot r_\ell~ entries},\underbrace{c_{2,3},\dots,c_{2,3}}_{r_2\cdot r_3~ entries},\dots, \underbrace{c_{2,\ell},\dots,c_{2,\ell}}_{r_2\cdot r_\ell~ entries},\\
\underbrace{c_{3,4},\dots,c_{3,4}}_{r_3\cdot r_4~ entries},\dots, \underbrace{c_{3,\ell},\dots,c_{3,\ell}}_{r_3\cdot r_\ell~ entries},\dots, \underbrace{c_{(\ell-1),\ell},\dots c_{(\ell-1),\ell}}_{r_{\ell-1}\cdot r_\ell~ entries})$.

Now consider the complete $\sum\limits_{i=1}^{\ell-1}\sum\limits_{j=i+1}^{\ell}r_i\cdot r_j$-partite graph and apply the same multiplicative rule of Iteration 1.

\item[(iii)] The number of vertices of the null graph is obtained recursively after exactly $t_\chi = (\ell -1)$ iterations.
\end{enumerate}	
\end{cor}
\begin{proof}
The proof of the result is straight forward from the recursion rule.
\end{proof}

We also observe that the maximum number of edges obtained through the iterations is equal to the number of vertices of the null graph. Furthermore, the maximum number of edges is obtained on the $t_\chi = (\ell-2)$-th iteration.

\section{On Chromatic Number of a Graph}

Inherent to a given colour colour cluster is that the number of vertices of the corresponding graphical embodiment is fixed at $\sum\limits_{i=1}^{\ell}|\mathcal{C}_i|$. In this section, we investigate chromatic blending for a given chromatic number of a graph $G$. It is known that a graph $G$ with $\chi(G) =2$ is a tree $T$ on at least two vertices. Therefore, any tree corresponds to some colour cluster $\C = (\mathcal{C}_1,\mathcal{C}_2)$, $|\mathcal{C}_1| \geq 1$, $|\mathcal{C}_2| \geq 1$. Hence, by Theorem \ref{Thm-3.1}, it follows that total chromatic blending is obtained after the iteration-$1$ and the number of vertices of the corresponding null graph is equal to the number of edges of the tree or put differently, it is equal to the number of vertices of the line graph $L(T)$. These observations lead us to the following result.

\begin{thm}\label{Thm-4.1}
Any graph that has chromatic number $\chi(G) = \ell,$ reaches total chromatic blending after $t_\chi = \ell-1$ chromatic blending iterations.
\end{thm}
\begin{proof}
Any graph $G$ that has $\chi(G) = \ell,$ has a corresponding colour cluster $\C = (\mathcal{C}_1,\mathcal{C}_2,\dots,\mathcal{C}_\ell)$, $|\mathcal{C}_i| \geq 1$, $\forall i$, $1\leq i \leq \ell$. Therefore, the result follows from Theorem \ref{Thm-3.1}.
\end{proof}

Note that Theorem \ref{Thm-4.1} holds for any loopless graph provided that each vertex $v$ has $d(v) \geq 1$.  The order of the maximum clique of a graph $G$ is its clique number, $\omega(G)$. Since $\chi(G) \geq \omega(G)$ (see \cite{AES}) and $\chi(G) \leq \Delta(G)+1,$ the trivial bounds for $t_\chi$ is given by $\omega(G)-1 \leq t_\chi \leq \Delta(G)$. It also follows that for any triangle-free graph $G$ with $\chi(G) =k$, the Mycielski graph of $G$, denoted by $\mu(G)$, has $t_\chi(\mu(G)) = k$. Furthermore, if $G$ is a graph and $H$ is any subgraph of $G$ then $t_\chi(H) \leq t_\chi(G)$. Finally, all results for $\chi(G)$ hold in terms of $t_\chi(G)$ by substituting $\chi(G) = t_\chi(G) +1$.

\section{Conclusion}

This paper serves as an introductory note on chromatic blending of colour clusters. In this paper, we proved that a graphical embodiment with minimum or maximum number of edges exists for any colour cluster. It will be interesting to find an elegant closed formula for the number of vertices of the null graph obtained from total chromatic blending.

The application of total chromatic blending is obtaining homogeneity amongst distinct objects through some process. Physical total colour blending is observed in many physical systems such as interpolation of electromagnetic waves (or rays), optical light, chemical bonding, mutation and other genetic engineering methods. Other more subtle social applications are political campaigning, religious campaigning, society structures to ensure orderliness and alike.

Chemical bonding can be modelled as colour blending. Each distinct chemical element is considered to be a colour and the chemical bond, the blend. Hence, in principle it is foreseen that this notion can find application in mathematical chemistry as well. This will open an interesting field of research.

\end{document}